\documentclass{amsart}
\usepackage{amsmath}
\usepackage{amsfonts}

\newtheorem{theorem}{Theorem}
\theoremstyle{plain}

\newtheorem{corollary}{Corollary}

\newtheorem{lemma}{Lemma}

\newtheorem{proposition}{Proposition}

\numberwithin{equation}{section}
\input{tcilatex}

\begin{document}
\title[Integral inequalities]{New estimates on generalization of some
integral inequalities for quasi-convex functions and their applications}
\author{\.{I}mdat \.{I}\c{s}can}
\address{Department of Mathematics, Faculty of Arts and Sciences,\\
Giresun University, 28100, Giresun, Turkey.}
\email{imdat.iscan@giresun.edu.tr}
\date{July 10, 2012}

\begin{abstract}
In this paper, we derive new estimates for the remainder term of the
midpoint, trapezoid, and Simpson formulae for functions whose derivatives in
absolute value at certain power are quasi-convex. Some applications to
special means of real numbers are also given.
\end{abstract}

\maketitle

\section{Introduction}

Let $f:I\subset \mathbb{R\rightarrow R}$ be a convex function defined on the
interval $I$ of real numbers and $a,b\in I$ with $a<b$. The following
inequality%
\begin{equation}
f\left( \frac{a+b}{2}\right) \leq \frac{1}{b-a}\dint\limits_{a}^{b}f(x)dx%
\leq \frac{f(a)+f(b)}{2}\text{.}  \label{1-1}
\end{equation}

holds. This double inequality is known in the literature as Hermite-Hadamard
integral inequality for convex functions.Note that some of the classical
inequalities for means can be derived from (\ref{1-1}) for appropriate
particular selections of the mapping $f$. Both inequalities hold in the
reversed direction if f is concave. See \cite{AD10,ADK10,ADK11,DP00,I07,K04}%
, the results of the generalization, improvement and extention of the famous
integral inequality (\ref{1-1}).

The notion of quasi-convex functions generalizes the notion of convex
functions. More precisely, a function $f:[a,b]\mathbb{\rightarrow R}$ is
said quasi-convex on $[a,b]$ if 
\begin{equation*}
f\left( \alpha x+(1-\alpha )y\right) \leq \sup \left\{ f(x),f(y)\right\} ,
\end{equation*}%
for any $x,y\in \lbrack a,b]$ and $\alpha \in \left[ 0,1\right] .$ Clearly,
any convex function is a quasi-convex function. Furthermore, there exist
quasi-convex functions which are not \ convex (see \cite{I07}).

The following inequality is well known in the literature as Simpson's
inequality .

Let $f:\left[ a,b\right] \mathbb{\rightarrow R}$ be a four times
continuously differentiable mapping on $\left( a,b\right) $ and $\left\Vert
f^{(4)}\right\Vert _{\infty }=\underset{x\in \left( a,b\right) }{\sup }%
\left\vert f^{(4)}(x)\right\vert <\infty .$ Then the following inequality
holds:%
\begin{equation*}
\left\vert \frac{1}{3}\left[ \frac{f(a)+f(b)}{2}+2f\left( \frac{a+b}{2}%
\right) \right] -\frac{1}{b-a}\dint\limits_{a}^{b}f(x)dx\right\vert \leq 
\frac{1}{2880}\left\Vert f^{(4)}\right\Vert _{\infty }\left( b-a\right) ^{2}.
\end{equation*}

\bigskip In recent years many authors have studied error estimations for
Simpson's inequality; for refinements, counterparts, generalizations and new
Simpson's type inequalities, see \cite{ADD09,AH11,SA11,SSO10,SSO10a}

In \cite{ADK10}, Alomari et al. established some upper bound for the right
-hand side of Hadamard's inequality for quasi-convex mappings, The authors
obtained the following results:

\begin{theorem}
Let $f:I\subset \mathbb{R\rightarrow R}$ be a differentiable mapping on $%
I^{\circ }$ such that $f^{\prime }\in L[a,b]$, where $a,b\in I$ with $a<b.$
If $\left\vert f^{\prime }\right\vert ^{p/(p-1)}$ is an quasi-convex on $%
[a,b]$, for $p>1,$ then the following inequality holds: 
\begin{eqnarray}
\left\vert \frac{f(a)+f(b)}{2}-\frac{1}{b-a}\dint\limits_{a}^{b}f(x)dx\right%
\vert &\leq &\frac{b-a}{4\left( p+1\right) ^{1/p}}\left[ \left( \sup \left\{
\left\vert f^{\prime }(\frac{a+b}{2})\right\vert ^{\frac{p}{p-1}},\left\vert
f^{\prime }(b)\right\vert ^{\frac{p}{p-1}}\right\} \right) ^{\frac{p-1}{p}%
}\right.  \notag \\
&&\left. +\left( \sup \left\{ \left\vert f^{\prime }(\frac{a+b}{2}%
)\right\vert ^{\frac{p}{p-1}},\left\vert f^{\prime }(a)\right\vert ^{\frac{p%
}{p-1}}\right\} \right) ^{\frac{p-1}{p}}\right] .  \label{1-4}
\end{eqnarray}
\end{theorem}

\begin{theorem}
Let $f:I^{\circ }\subset \mathbb{R\rightarrow R}$ be a differentiable
mapping on $I^{\circ },$ $a,b\in I^{\circ }$ with $a<b.$ If $\left\vert
f^{\prime }\right\vert ^{q}$ is an quasi-convex on $[a,b]$, for $q\geq 1,$
then the following inequality holds:%
\begin{eqnarray}
\left\vert \frac{f(a)+f(b)}{2}-\frac{1}{b-a}\dint\limits_{a}^{b}f(x)dx\right%
\vert &\leq &\frac{b-a}{8}\left[ \left( \sup \left\{ \left\vert f^{\prime }(%
\frac{a+b}{2})\right\vert ^{q},\left\vert f^{\prime }(b)\right\vert
^{q}\right\} \right) ^{\frac{1}{q}}\right.  \notag \\
&&\left. +\left( \sup \left\{ \left\vert f^{\prime }(\frac{a+b}{2}%
)\right\vert ^{q},\left\vert f^{\prime }(a)\right\vert ^{q}\right\} \right)
^{\frac{1}{q}}\right] .  \label{1-5}
\end{eqnarray}
\end{theorem}

In this paper, in order to provide a unified approach to establish midpoint
inequality, trapezoid inequality and Simpson's inequality for functions
whose derivatives in absolute value at certain power are quasi-convex, we
need the following lemma given by Iscan in \cite{I12}:

\begin{lemma}
\label{1.1}Let $f:I\subset \mathbb{R\rightarrow R}$ be a differentiable
mapping on $I^{\circ }$ such that $f^{\prime }\in L[a,b]$, where $a,b\in I$
with $a<b$ and $\theta ,\lambda \in \left[ 0,1\right] $ , then the following
equality holds:%
\begin{eqnarray}
&&\left( 1-\theta \right) \left( \lambda f(a)+\left( 1-\lambda \right)
f(b)\right) +\theta f(\left( 1-\lambda \right) a+\lambda b)-\frac{1}{b-a}%
\dint\limits_{a}^{b}f(x)dx  \label{1-6} \\
&=&\left( b-a\right) \left[ -\lambda ^{2}\dint\limits_{0}^{1}\left( t-\theta
\right) f^{\prime }\left( ta+\left( 1-t\right) \left[ \left( 1-\lambda
\right) a+\lambda b\right] \right) dt\right.  \notag \\
&&\left. +\left( 1-\lambda \right) ^{2}\dint\limits_{0}^{1}\left( t-\theta
\right) f^{\prime }\left( tb+\left( 1-t\right) \left[ \left( 1-\lambda
\right) a+\lambda b\right] \right) dt\right] .  \notag
\end{eqnarray}
\end{lemma}

\section{Main results}

\begin{theorem}
\label{2.1}Let $f:I\subset \mathbb{R\rightarrow R}$ be a differentiable
mapping on $I^{\circ }$ such that $f^{\prime }\in L[a,b]$, where $a,b\in
I^{\circ }$ with $a<b$ and $\alpha ,\lambda \in \left[ 0,1\right] $. If $%
\left\vert f^{\prime }\right\vert ^{q}$ is quasi-convex on $[a,b]$, $q\geq
1, $ then the following inequality holds:%
\begin{eqnarray}
&&\left\vert \left( 1-\theta \right) \left( \lambda f(a)+\left( 1-\lambda
\right) f(b)\right) +\theta f(\left( 1-\lambda \right) a+\lambda b)-\frac{1}{%
b-a}\dint\limits_{a}^{b}f(x)dx\right\vert  \notag \\
&\leq &\left( b-a\right) \left( \theta ^{2}-\theta +\frac{1}{2}\right) \left[
\lambda ^{2}\left( \sup \left\{ \left\vert f^{\prime }(a)\right\vert
^{q},\left\vert f^{\prime }(C)\right\vert ^{q}\right\} \right) ^{\frac{1}{q}%
}\right.  \notag \\
&&\left. +\left( 1-\lambda \right) ^{2}\left( \sup \left\{ \left\vert
f^{\prime }(b)\right\vert ^{q},\left\vert f^{\prime }(C)\right\vert
^{q}\right\} \right) ^{\frac{1}{q}}\right]  \label{2-1}
\end{eqnarray}%
where $C=\left( 1-\lambda \right) a+\lambda b$ $.$
\end{theorem}

\begin{proof}
Suppose that $q\geq 1$ and $C=\left( 1-\lambda \right) a+\lambda b.$ From
Lemma \ref{2.1} and using the well known power mean inequality, we have%
\begin{eqnarray*}
&&\left\vert \left( 1-\theta \right) \left( \lambda f(a)+\left( 1-\lambda
\right) f(b)\right) +\theta f(C)-\frac{1}{b-a}\dint\limits_{a}^{b}f(x)dx%
\right\vert \leq \left( b-a\right) \\
&&\left[ \lambda ^{2}\dint\limits_{0}^{1}\left\vert t-\theta \right\vert
\left\vert f^{\prime }\left( ta+\left( 1-t\right) C\right) \right\vert
dt+\left( 1-\lambda \right) ^{2}\dint\limits_{0}^{1}\left\vert t-\theta
\right\vert \left\vert f^{\prime }\left( tb+\left( 1-t\right) C\right)
\right\vert dt\right] \\
&\leq &\left( b-a\right) \left\{ \lambda ^{2}\left(
\dint\limits_{0}^{1}\left\vert t-\theta \right\vert dt\right) ^{1-\frac{1}{q}%
}\left( \dint\limits_{0}^{1}\left\vert t-\theta \right\vert \left\vert
f^{\prime }\left( ta+\left( 1-t\right) C\right) \right\vert ^{q}dt\right) ^{%
\frac{1}{q}}\right.
\end{eqnarray*}%
\begin{equation}
\left. +\left( 1-\lambda \right) ^{2}\left( \dint\limits_{0}^{1}\left\vert
t-\theta \right\vert dt\right) ^{1-\frac{1}{q}}\left(
\dint\limits_{0}^{1}\left\vert t-\theta \right\vert \left\vert f^{\prime
}\left( tb+\left( 1-t\right) C\right) \right\vert ^{q}dt\right) ^{\frac{1}{q}%
}\right\} .  \label{2-1a}
\end{equation}

Since $\left\vert f^{\prime }\right\vert ^{q}$ is quasi-convex on $[a,b],$
we know that for $t\in \left[ 0,1\right] $%
\begin{equation*}
\left\vert f^{\prime }\left( ta+(1-t)C\right) \right\vert ^{q}\leq \sup
\left\{ \left\vert f^{\prime }(a)\right\vert ^{q},\left\vert f^{\prime
}(C)\right\vert ^{q}\right\} ,
\end{equation*}%
and%
\begin{equation*}
\left\vert f^{\prime }\left( tb+\left( 1-t\right) C\right) \right\vert
^{q}\leq \sup \left\{ \left\vert f^{\prime }(a)\right\vert ^{q},\left\vert
f^{\prime }(C)\right\vert ^{q}\right\} .
\end{equation*}%
Hence, by simple computation%
\begin{equation}
\dint\limits_{0}^{1}\left\vert t-\theta \right\vert dt=\theta ^{2}-\theta +%
\frac{1}{2},  \label{2-1b}
\end{equation}%
\begin{equation}
\dint\limits_{0}^{1}\left\vert t-\theta \right\vert \left\vert f^{\prime
}\left( ta+\left( 1-t\right) C\right) \right\vert ^{q}dt=\left( \theta
^{2}-\theta +\frac{1}{2}\right) \sup \left\{ \left\vert f^{\prime
}(a)\right\vert ^{q},\left\vert f^{\prime }(C)\right\vert ^{q}\right\} ,
\label{2-1c}
\end{equation}%
and%
\begin{equation}
\dint\limits_{0}^{1}\left\vert t-\theta \right\vert \left\vert f^{\prime
}\left( tb+\left( 1-t\right) C\right) \right\vert ^{q}dt=\left( \theta
^{2}-\theta +\frac{1}{2}\right) \sup \left\{ \left\vert f^{\prime
}(b)\right\vert ^{q},\left\vert f^{\prime }(C)\right\vert ^{q}\right\} .
\label{2-1d}
\end{equation}%
Thus, using (\ref{2-1b})-(\ref{2-1d}) in (\ref{2-1a}), we obtain the
inequality (\ref{2-1}). This completes the proof.
\end{proof}

\begin{corollary}
Under the assumptions of Theorem \ref{2.1} with $q=1,$ the inequality (\ref%
{2-1}) reduced to the following inequality%
\begin{eqnarray*}
&&\left\vert \left( 1-\theta \right) \left( \lambda f(a)+\left( 1-\lambda
\right) f(b)\right) +\theta f(\left( 1-\lambda \right) a+\lambda b)-\frac{1}{%
b-a}\dint\limits_{a}^{b}f(x)dx\right\vert \\
&\leq &\left( b-a\right) \left( \theta ^{2}-\theta +\frac{1}{2}\right) \left[
\lambda ^{2}\sup \left\{ \left\vert f^{\prime }(a)\right\vert ,\left\vert
f^{\prime }(C)\right\vert \right\} \right. \\
&&\left. +\left( 1-\lambda \right) ^{2}\sup \left\{ \left\vert f^{\prime
}(b)\right\vert ,\left\vert f^{\prime }(C)\right\vert \right\} \right] .
\end{eqnarray*}
\end{corollary}

\begin{corollary}
Under the assumptions of Theorem \ref{2.1} with $\lambda =\frac{1}{2}$ and $%
\theta =\frac{2}{3}$, from the inequality (\ref{2-1}) we get the following
Simpson type inequality 
\begin{eqnarray*}
&&\left\vert \frac{1}{6}\left[ f(a)+4f\left( \frac{a+b}{2}\right) +f(b)%
\right] -\frac{1}{b-a}\dint\limits_{a}^{b}f(x)dx\right\vert \\
&\leq &\left( b-a\right) \left( \frac{5}{72}\right) \left[ \left( \sup
\left\{ \left\vert f^{\prime }(a)\right\vert ^{q},\left\vert f^{\prime
}\left( \frac{a+b}{2}\right) \right\vert ^{q}\right\} \right) ^{\frac{1}{q}%
}+\left( \sup \left\{ \left\vert f^{\prime }(b)\right\vert ^{q},\left\vert
f^{\prime }\left( \frac{a+b}{2}\right) \right\vert ^{q}\right\} \right) ^{%
\frac{1}{q}}\right] .
\end{eqnarray*}
\end{corollary}

\begin{corollary}
Under the assumptions of Theorem \ref{2.1} with $\lambda =\frac{1}{2}$ and $%
\theta =1,$from the inequality (\ref{2-1}) we get the following midpoint
inequality%
\begin{eqnarray*}
&&\left\vert f\left( \frac{a+b}{2}\right) -\frac{1}{b-a}\dint%
\limits_{a}^{b}f(x)dx\right\vert \\
&\leq &\frac{b-a}{8}\left[ \left( \sup \left\{ \left\vert f^{\prime
}(a)\right\vert ^{q},\left\vert f^{\prime }\left( \frac{a+b}{2}\right)
\right\vert ^{q}\right\} \right) ^{\frac{1}{q}}+\left( \sup \left\{
\left\vert f^{\prime }(b)\right\vert ^{q},\left\vert f^{\prime }\left( \frac{%
a+b}{2}\right) \right\vert ^{q}\right\} \right) ^{\frac{1}{q}}\right] .
\end{eqnarray*}
\end{corollary}

\begin{corollary}
Under the assumptions of Theorem \ref{2.1} with $\lambda =\frac{1}{2}$ and $%
\theta =0,$from the inequality (\ref{2-1}) we get the following trapezoid
inequality%
\begin{eqnarray*}
&&\left\vert \frac{f\left( a\right) +f\left( b\right) }{2}-\frac{1}{b-a}%
\dint\limits_{a}^{b}f(x)dx\right\vert \\
&\leq &\frac{b-a}{8}\left[ \left( \sup \left\{ \left\vert f^{\prime
}(a)\right\vert ^{q},\left\vert f^{\prime }\left( \frac{a+b}{2}\right)
\right\vert ^{q}\right\} \right) ^{\frac{1}{q}}+\left( \sup \left\{
\left\vert f^{\prime }(b)\right\vert ^{q},\left\vert f^{\prime }\left( \frac{%
a+b}{2}\right) \right\vert ^{q}\right\} \right) ^{\frac{1}{q}}\right] .
\end{eqnarray*}%
which is the same of the inequality (\ref{1-5}).
\end{corollary}

Using Lemma \ref{1.1} we shall give another result for convex functions as
follows.

\begin{theorem}
\label{2.2}Let $f:I\subset \mathbb{R\rightarrow R}$ be a differentiable
mapping on $I^{\circ }$ such that $f^{\prime }\in L[a,b]$, where $a,b\in
I^{\circ }$ with $a<b$ and $\alpha ,\lambda \in \left[ 0,1\right] $. If $%
\left\vert f^{\prime }\right\vert ^{q}$ is quasi-convex on $[a,b]$, $q>1,$
then the following inequality holds:%
\begin{eqnarray}
&&\left\vert \left( 1-\theta \right) \left( \lambda f(a)+\left( 1-\lambda
\right) f(b)\right) +\theta f(\left( 1-\lambda \right) a+\lambda b)-\frac{1}{%
b-a}\dint\limits_{a}^{b}f(x)dx\right\vert  \label{2-2} \\
&\leq &\left( b-a\right) \left( \frac{\theta ^{p+1}+\left( 1-\theta \right)
^{p+1}}{p+1}\right) ^{\frac{1}{p}}\left[ \lambda ^{2}\left( \sup \left\{
\left\vert f^{\prime }(a)\right\vert ^{q},\left\vert f^{\prime
}(C)\right\vert ^{q}\right\} \right) ^{\frac{1}{q}}\right.  \notag \\
&&\left. +\left( 1-\lambda \right) ^{2}\left( \sup \left\{ \left\vert
f^{\prime }(b)\right\vert ^{q},\left\vert f^{\prime }(C)\right\vert
^{q}\right\} \right) ^{\frac{1}{q}}\right]  \notag
\end{eqnarray}%
where $C=\left( 1-\lambda \right) a+\lambda b$ and $\frac{1}{p}+\frac{1}{q}%
=1.$
\end{theorem}

\begin{proof}
Suppose that $C=\left( 1-\lambda \right) a+\lambda b.$ From Lemma \ref{2.1}
and by H\"{o}lder's integral inequality, we have%
\begin{eqnarray*}
&&\left\vert \left( 1-\theta \right) \left( \lambda f(a)+\left( 1-\lambda
\right) f(b)\right) +\theta f(C)-\frac{1}{b-a}\dint\limits_{a}^{b}f(x)dx%
\right\vert \leq \left( b-a\right) \\
&&\left[ \lambda ^{2}\dint\limits_{0}^{1}\left\vert t-\theta \right\vert
\left\vert f^{\prime }\left( ta+\left( 1-t\right) C\right) \right\vert
dt+\left( 1-\lambda \right) ^{2}\dint\limits_{0}^{1}\left\vert t-\theta
\right\vert \left\vert f^{\prime }\left( tb+\left( 1-t\right) C\right)
\right\vert dt\right] \\
&\leq &\left( b-a\right) \left\{ \lambda ^{2}\left(
\dint\limits_{0}^{1}\left\vert t-\theta \right\vert ^{p}dt\right) ^{\frac{1}{%
p}}\left( \dint\limits_{0}^{1}\left\vert f^{\prime }\left( ta+\left(
1-t\right) C\right) \right\vert ^{q}dt\right) ^{\frac{1}{q}}\right.
\end{eqnarray*}%
\begin{equation}
\left. +\left( 1-\lambda \right) ^{2}\left( \dint\limits_{0}^{1}\left\vert
t-\theta \right\vert ^{p}dt\right) ^{\frac{1}{p}}\left(
\dint\limits_{0}^{1}\left\vert f^{\prime }\left( tb+\left( 1-t\right)
C\right) \right\vert ^{q}dt\right) ^{\frac{1}{q}}\right\} .  \label{2-2a}
\end{equation}%
Since $\left\vert f^{\prime }\right\vert ^{q}$ is quasi-convex on $[a,b]$,
we get 
\begin{equation}
\dint\limits_{0}^{1}\left\vert f^{\prime }\left( ta+\left( 1-t\right)
C\right) \right\vert ^{q}dt=\sup \left\{ \left\vert f^{\prime
}(a)\right\vert ^{q},\left\vert f^{\prime }(C)\right\vert ^{q}\right\}
\label{2-2b}
\end{equation}%
Similarly, 
\begin{equation}
\dint\limits_{0}^{1}\left\vert f^{\prime }\left( tb+\left( 1-t\right)
C\right) \right\vert ^{q}dt=\sup \left\{ \left\vert f^{\prime
}(b)\right\vert ^{q},\left\vert f^{\prime }(C)\right\vert ^{q}\right\} .
\label{2-2c}
\end{equation}%
By simple computation%
\begin{equation}
\dint\limits_{0}^{1}\left\vert t-\theta \right\vert ^{p}dt=\frac{\theta
^{p+1}+\left( 1-\theta \right) ^{p+1}}{p+1},  \label{2-2d}
\end{equation}%
thus, using (\ref{2-2b})-(\ref{2-2d}) in (\ref{2-2a}), we obtain the
inequality (\ref{2-2}). This completes the proof.
\end{proof}

\begin{corollary}
Under the assumptions of Theorem \ref{2.2} with $\lambda =\frac{1}{2}$ and $%
\theta =\frac{2}{3}$, from the inequality (\ref{2-2}) we get the following
Simpson type inequality 
\begin{equation*}
\left\vert \frac{1}{6}\left[ f(a)+4f\left( \frac{a+b}{2}\right) +f(b)\right]
-\frac{1}{b-a}\dint\limits_{a}^{b}f(x)dx\right\vert
\end{equation*}%
\begin{eqnarray*}
&\leq &\frac{b-a}{12}\left( \frac{1+2^{p+1}}{3\left( p+1\right) }\right) ^{%
\frac{1}{p}}\left\{ \left( \sup \left\{ \left\vert f^{\prime }\left( \frac{%
a+b}{2}\right) \right\vert ^{q},\left\vert f^{\prime }(a)\right\vert
^{q}\right\} \right) ^{\frac{1}{q}}\right. \\
&&\left. +\left( \sup \left\{ \left\vert f^{\prime }\left( \frac{a+b}{2}%
\right) \right\vert ^{q},\left\vert f^{\prime }(b)\right\vert ^{q}\right\}
\right) ^{\frac{1}{q}}\right\} .
\end{eqnarray*}
\end{corollary}

\begin{corollary}
Under the assumptions of Theorem \ref{2.2}with $\lambda =\frac{1}{2}$ and $%
\theta =0$, from the inequality (\ref{2-2}) we get the following trapezoid
inequality%
\begin{eqnarray*}
\left\vert \frac{f(a)+f(b)}{2}-\frac{1}{b-a}\dint\limits_{a}^{b}f(x)dx\right%
\vert &\leq &\frac{b-a}{4\left( p+1\right) ^{1/p}}\left[ \left\{ \left( \sup
\left\{ \left\vert f^{\prime }\left( \frac{a+b}{2}\right) \right\vert
^{q},\left\vert f^{\prime }(a)\right\vert ^{q}\right\} \right) ^{\frac{1}{q}%
}\right. \right. \\
&&\left. +\left\{ \left( \sup \left\{ \left\vert f^{\prime }\left( \frac{a+b%
}{2}\right) \right\vert ^{q},\left\vert f^{\prime }(b)\right\vert
^{q}\right\} \right) ^{\frac{1}{q}}\right. \right] .
\end{eqnarray*}%
which is the same of the inequality (\ref{1-4}).
\end{corollary}

\begin{corollary}
Under the assumptions of Theorem \ref{2.2} with $\lambda =\frac{1}{2}$ and $%
\theta =1$, from the inequality (\ref{2-2}) we get the following midpoint
inequality%
\begin{eqnarray*}
\left\vert f\left( \frac{a+b}{2}\right) -\frac{1}{b-a}\dint%
\limits_{a}^{b}f(x)dx\right\vert &\leq &\frac{b-a}{4\left( p+1\right) ^{1/p}}%
\left[ \left\{ \left( \sup \left\{ \left\vert f^{\prime }\left( \frac{a+b}{2}%
\right) \right\vert ^{q},\left\vert f^{\prime }(a)\right\vert ^{q}\right\}
\right) ^{\frac{1}{q}}\right. \right. \\
&&\left. +\left\{ \left( \sup \left\{ \left\vert f^{\prime }\left( \frac{a+b%
}{2}\right) \right\vert ^{q},\left\vert f^{\prime }(b)\right\vert
^{q}\right\} \right) ^{\frac{1}{q}}\right. \right] ,
\end{eqnarray*}%
which is the better than the inequality in \cite[Corollary 8]{AD10}.
\end{corollary}

\section{Some applications for special means}

Let us recall the following special means of arbitrary real numbers $a,b$
with $a\neq b$ and $\alpha \in \left[ 0,1\right] :$

\begin{enumerate}
\item The weighted arithmetic mean%
\begin{equation*}
A_{\alpha }\left( a,b\right) :=\alpha a+(1-\alpha )b,~a,b\in 
\mathbb{R}
.
\end{equation*}

\item The unweighted arithmetic mean%
\begin{equation*}
A\left( a,b\right) :=\frac{a+b}{2},~a,b\in 
\mathbb{R}
.
\end{equation*}

\item The weighted harmonic mean%
\begin{equation*}
H_{\alpha }\left( a,b\right) :=\left( \frac{\alpha }{a}+\frac{1-\alpha }{b}%
\right) ^{-1},\ \ a,b\in 
\mathbb{R}
\backslash \left\{ 0\right\} .
\end{equation*}

\item The unweighted harmonic mean%
\begin{equation*}
H\left( a,b\right) :=\frac{2ab}{a+b},\ \ a,b\in 
\mathbb{R}
\backslash \left\{ 0\right\} .
\end{equation*}

\item The Logarithmic mean%
\begin{equation*}
L\left( a,b\right) :=\frac{b-a}{\ln b-\ln a},\ \ a,b>0,\ a\neq b\ .
\end{equation*}

\item Then n-Logarithmic mean%
\begin{equation*}
L_{n}\left( a,b\right) :=\ \left( \frac{b^{n+1}-a^{n+1}}{(n+1)(b-a)}\right)
^{\frac{1}{n}}\ ,\ n\in 
\mathbb{N}
,\ a,b\in 
\mathbb{R}
,\ a\neq b.
\end{equation*}
\end{enumerate}

\begin{proposition}
Let $a,b\in 
\mathbb{R}
$ with $a<b,$ and $n\in 
\mathbb{N}
,\ n\geq 2.$ Then, for $\theta ,\lambda \in \left[ 0,1\right] $ and $q\geq
1, $we have the following inequality:%
\begin{eqnarray*}
&&\left\vert \left( 1-\theta \right) A_{\lambda }\left( a^{n},b^{n}\right)
+\theta A_{\lambda }^{n}\left( a,b\right) -L_{n}^{n}\left( a,b\right)
\right\vert \\
&\leq &\left( b-a\right) \left( \theta ^{2}-\theta +\frac{1}{2}\right) n%
\left[ \lambda ^{2}\left( \sup \left\{ \left\vert a\right\vert
^{(n-1)q},\left\vert A_{\lambda }\left( b,a\right) \right\vert
^{(n-1)q}\right\} \right) ^{\frac{1}{q}}\right. \\
&&\left. +\left( 1-\lambda \right) ^{2}\left( \sup \left\{ \left\vert
b\right\vert ^{(n-1)q},\left\vert A_{\lambda }\left( b,a\right) \right\vert
^{(n-1)q}\right\} \right) ^{\frac{1}{q}}\right] .
\end{eqnarray*}
\end{proposition}

\begin{proof}
The assertion follows from Theorem \ref{2.1}, for$\ f(x)=x^{n},\ x\in 
\mathbb{R}
.$
\end{proof}

\begin{proposition}
Let $a,b\in 
\mathbb{R}
$ with $a<b,$ and $n\in 
\mathbb{N}
,\ n\geq 2.$ Then, for $\theta ,\lambda \in \left[ 0,1\right] $ and $q>1,$we
have the following inequality:%
\begin{eqnarray*}
&&\left\vert \left( 1-\theta \right) A_{\lambda }\left( a^{n},b^{n}\right)
+\theta A_{\lambda }^{n}\left( a,b\right) -L_{n}^{n}\left( a,b\right)
\right\vert \\
&\leq &\left( b-a\right) \left( \frac{\theta ^{p+1}+\left( 1-\theta \right)
^{p+1}}{p+1}\right) ^{\frac{1}{p}}n\left[ \lambda ^{2}\left( \sup \left\{
\left\vert a\right\vert ^{(n-1)q},\left\vert A_{\lambda }\left( b,a\right)
\right\vert ^{q}\right\} \right) ^{\frac{1}{q}}\right. \\
&&\left. +\left( 1-\lambda \right) ^{2}\left( \sup \left\{ \left\vert
b\right\vert ^{(n-1)q},\left\vert A_{\lambda }\left( b,a\right) \right\vert
^{q}\right\} \right) ^{\frac{1}{q}}\right]
\end{eqnarray*}%
where $\frac{1}{p}+\frac{1}{q}=1.$
\end{proposition}

\begin{proof}
The assertion follows from Theorem \ref{2.2}, for$\ f(x)=x^{n},\ x\in 
\mathbb{R}
.$
\end{proof}

\begin{proposition}
Let $a,b\in 
\mathbb{R}
$ with $0<a<b,$ and $\theta ,\lambda \in \left[ 0,1\right] $ . Then, for $%
q\geq 1,$ we have the following inequality:%
\begin{eqnarray*}
&&\left\vert \left( 1-\theta \right) H_{\lambda }^{-1}\left( a,b\right)
+\theta A_{\lambda }^{-1}\left( a,b\right) -L^{-1}\left( a,b\right)
\right\vert \\
&\leq &\left( b-a\right) \left( \theta ^{2}-\theta +\frac{1}{2}\right) \left[
\lambda ^{2}\left( \sup \left\{ a^{-2q},A_{\lambda }\left( b,a\right)
^{-2q}\right\} \right) ^{\frac{1}{q}}\right. \\
&&\left. +\left( 1-\lambda \right) ^{2}\left( \sup \left\{
b^{-2q},A_{\lambda }\left( b,a\right) ^{-2q}\right\} \right) ^{\frac{1}{q}}%
\right] .
\end{eqnarray*}
\end{proposition}

\begin{proof}
The assertion follows from Theorem \ref{2.1}., for$\ f(x)=\frac{1}{x},\ x\in
\left( 0,\infty \right) .$
\end{proof}

\begin{proposition}
Let $a,b\in 
\mathbb{R}
$ with $0<a<b,$ and $\theta ,\lambda \in \left[ 0,1\right] $. Then, for $%
q>1, $ we have the following inequality:%
\begin{eqnarray*}
&&\left\vert \left( 1-\theta \right) H_{\lambda }^{-1}\left( a,b\right)
+\theta A_{\lambda }^{-1}\left( a,b\right) -L^{-1}\left( a,b\right)
\right\vert \\
&\leq &\left( b-a\right) \left( \frac{\theta ^{p+1}+\left( 1-\theta \right)
^{p+1}}{p+1}\right) ^{\frac{1}{p}}\left[ \lambda ^{2}\left( \sup \left\{
a^{-2q},A_{\lambda }\left( b,a\right) ^{-2q}\right\} \right) ^{\frac{1}{q}%
}\right. \\
&&\left. +\left( 1-\lambda \right) ^{2}\left( \sup \left\{
b^{-2q},A_{\lambda }\left( b,a\right) ^{-2q}\right\} \right) ^{\frac{1}{q}}%
\right]
\end{eqnarray*}
\end{proposition}

\begin{proof}
The assertion follows from Theorem \ref{2.2}, for$\ f(x)=\frac{1}{x}$,$\
x\in \left( 0,\infty \right) .$
\end{proof}

\end{document}